\documentclass[reqno]{amsart}

\usepackage[all]{xy}
\usepackage{textcomp,gensymb}
\usepackage{eurosym}
\usepackage{enumerate}

\usepackage{graphicx,graphics} 
\DeclareGraphicsExtensions{.pdf}
\usepackage{caption}
\usepackage{subcaption}
\usepackage{amsmath,amsthm, amscd, amssymb, amsfonts, mathrsfs} 
\usepackage[mathscr]{euscript} 
\usepackage[normalem]{ulem}
\newtheorem{theorem}{Theorem}[section]
\newtheorem{corollary}[theorem]{Corollary}
\newtheorem{lemma}[theorem]{Lemma}
\newtheorem{proposition}[theorem]{Proposition} 
 
\theoremstyle{definition}

\theoremstyle{remark}

\newtheoremstyle{mystyle}{2mm}{0mm}{}{}{\bfseries}{}{1.1ex}{\thmnumber{#2}.\hspace*{1ex}\thmnote{#3}}
\theoremstyle{mystyle}
\newtheorem{fact}[theorem]{}
\newtheoremstyle{myremark}{2mm}{0mm}{}{}{\bfseries}{}{ }{\thmname{#1}
\thmnumber{#2}. \thmnote{ #3}}
\theoremstyle{myremark}

\newcommand\ot{\otimes}

\newcommand\op{\oplus}
\newcommand\kal{\mathscr}
\newcommand\bb{\mathbb}

\newcommand\dr{\mathrm}
\newcommand\ds{\displaystyle}

\DeclareMathOperator{\lgr}{{^{\ast\text{-gr}}}}
\DeclareMathOperator{\lgrr}{^{\ast\text{-gr}}\!}

\def\T3{\mathbb{T}_3}

\newcommand{\xrar}{\xrightarrow}
\newcommand{\seq}{\subseteq}

\DeclareMathAlphabet{\mathpzc}{OT1}{pzc}{m}{it}

\usepackage[colorlinks, citecolor=blue]{hyperref}

\begin{document}

\title{The EXT Ring of Koszul Rings}
\author{Adrian Manea}

\address{University of Bucharest, Department of Mathematics, 14 Academiei Street, Bucharest, Ro-010014, Romania}
\date{}

\begin{abstract}
The aim of this article is to study the $\dr{Ext}$ ring associated to a Koszul $R$-ring and to use it to provide further characterisations of the latter. As such, for $R$ being a semisimple ring and $A$ a graded Koszul $R$-ring, we will prove that there is an isomorphism of DG rings between $\kal{E}(A):=\dr{Ext}^\bullet_A(R,R)$ and $\lgr\dr{T}(A) \simeq \dr{E}(\lgrr A)$. Also, the $\dr{Ext}$ $R$-ring will prove to be isomorphic to the shriek ring of the left graded dual of $A$, namely $\kal{E}(A) \simeq (\lgrr A)^!$. As an application, these isomorphisms will be studied in the context of incidence $R$-(co)rings for Koszul posets. Thus, we will obtain a description and method of computing the shriek ring for $\Bbbk^c[\kal{P}]$, the incidence $R$-coring of a Koszul poset. Another application is provided for monoid rings associated to submonoids of $\bb{Z}^n$.
\end{abstract}

  \maketitle
  
\section*{Introduction}
In \cite{jps}, the authors jointly developed an approach for studying Koszul $R$-rings using what was called \emph{Koszul pairs}, a construction which we briefly recall. Let $R$ be a semisimple ring. By an $R$-ring is indicated an algebra in the tensor category of $R$-bimodules. Let $A$ be a graded and connected $R$-ring, $A=\ds\oplus_{n \geq 0} \ A^n$, with $A^0=R$ and $C$ be a graded and connected $R$-coring, $C=\ds\oplus_{n \geq 0} \ C_n$, with $C_0=R$. The pair $(A,C)$ is called an \emph{almost-Koszul pair} if there exists an isomorphism of $R$-bimodules $\theta_{C,A}: C_1 \to A^1$ such that the following composition gives the zero map:
\begin{equation} \label{ec:theta}
C_2 \xrar{\Delta_{1,1}} C_1 \ot C_1 \xrar{\theta_{C,A} \ot \theta_{C,A}} A^1 \ot A^1 \xrar{\mu^{1,1}} A^2.
\end{equation}

Note that $\Delta_{1,1}$ is the component of the comultiplication for $C$, the map $\mu^{1,1}$ is the corresponding component of the multiplication for $A$ and any unadorned tensor products are considered over $R$.

The study of almost Koszul pairs is continued by associating to them three cochain complexes and three chain complexes, situated in different categories and which measure how far is $A$ from being a Koszul ring. Theorem 2.3 in \cite{jps} shows that if any of the six complexes is exact, then all of them are so, case in which the pair $(A,C)$ is called a \emph{Koszul pair}. It is also proved in the cited article that this approach agrees with the classical one from \cite{bgs}, in the sense that a Koszul pair consists of a Koszul $R$-ring and a Koszul $R$-coring. The latter structure was introduced and studied at length as a natural dualisation of the former in \cite[\S 3]{ms}.

To the components of any (almost) Koszul pair one associates a graded , connected coring and a similar ring. Namely, define $T(A):= \dr{Tor}_\bullet^A(R,R)$, which is a connected, graded coring and $E(C):= \dr{Ext}^\bullet_C(R,R)$, which is a connected, graded $R$-ring. Note that, in general, the construction of $E(C)$ and $T(A)$ works for arbitrary graded, connected (co)rings. The properties and the interplay between the pair $(A,C)$ and these two structures are studied at length in \cite{jps} and \cite{ms}. 

Another structure that is of great use is the \emph{graded linear dual} of a left locally finite $R$-(co)ring. For such a structure $A$, the left graded linear dual is defined taking into consideration left $R$-module structures as: 
\[\lgrr A = \ds\op_{n \geq 0}  {^\ast\!}(A^n) = \ds\op_{n \geq 0} \ \dr{Hom}_R(_RA,_RR).\]

The right version is defined analogously. Dually, for a left locally finite connected $R$-coring $C$, one can define $\lgr C$ and prove that $\lgrr A$ becomes a graded, connected $R^{op}$-coring and $\lgrr C$ becomes a graded, connected $R^{op}$-ring. The properties and some basic results involving the graded linear duals are presented in \cite[\S 4.1-4.2]{ms}. Also, we remark that the notion of a left (right) dual for a left (right) finitely generated $R$-ring was first introduced in \cite[\S 17.9]{brw}.

In this article, for a Koszul pair $(A,C)$, we are interested in studying the connections between the $R$-coring $T(A)$ and the $R$-ring $E(C)$, on the one hand, and the $\dr{Ext}$ $R$-ring $\kal{E}(A) := \dr{Ext}_A^\bullet(R,R)$, on the other. This will be computed using the left normalised bar resolution of $A$, denoted by $\beta_\bullet^l(A)$, which we will compare with the Koszul resolution $\dr{K}_\bullet^l(A,C)$. We will prove that there is a morphism $\varphi : \dr{K}_\bullet^l(A,C) \to \beta_\bullet^l(A)$, which is a quasi-isomorphism and thus induces an isomorphism of rings from $\kal{E}(A)$ to $\lgr T(A)$. Further, using \cite[Corollary 4.7]{ms}, they are also isomorphic to $E(\lgrr A)$.

As an application, we study the particular case of Koszul posets. Thus, if $\kal{P}$ is a finite, graded poset which is Koszul (see \cite[\S 4.8]{ms}) and $A := \Bbbk^a[\kal{P}]$ is its incidence $R$-ring, then $\lgrr A \simeq \Bbbk^c[\kal{P}]$, cf. \cite[Theorem 4.9]{ms}. Moreover, in this particular case, $\kal{E}(A) \simeq \Bbbk^c[\kal{P}]^!$, the shriek ring associated to the incidence $R$-coring for the poset $\kal{P}$. This isomorphism will give us an easier way of computing the former structure and concrete examples will be provided in the case of Koszul posets. Furthermore, we discuss an example in the direction of monoid rings associated to submonoids of $\bb{Z}^n$. The Koszulity of the respective monoid ring will be proved using tools from commutative algebra. Also, the results developed in the second part of the article will allow a concrete description of the $\dr{Ext}$ algebra.

\section{Preliminaries}
We will present here some results and definitions that are needed later on in the article. They will be heavily based on \cite{jps} and \cite{ms}, so we indicate these articles as comprehensive reads for the notions that we introduce and use.

\begin{fact}[Notation.] Some shorthand notations are used throughout the article and we collect them here. For the graded, connected $R$-ring $A=\ds\op_{n \geq 0}\ A^n$, we will denote by $A_+ := \ds\op_{n > 0} A^n$. We can further make the identification $A_+ = A/A^0$. The same notation will be used for the graded, connected $R$-coring $C=\ds\op_{n \geq 0} \ C_n$, hence we have $C_+ = C/C_0$. Note that the (co)multiplications of $A$ and $C$ respectively induce maps $\mu_+ : A_+ \ot A_+ \to A_+$ and $\Delta_+ : C_+ \to C_+ \ot C_+$ respectively.

The $n$-fold tensor product of $A$ with itself will be denoted by $A^{(n)}$. 

Given that we are working with graded (co)rings, it is worth making a notational remark regarding Sweedler's notation for comultiplication and a respective one for multiplication. As such, for the $R$-ring $A$, the multiplication map $\mu : A \ot A \to A$ has, for any $m,n \geq 0$, the components $\mu^{m,n} : A^m \ot A^n \to A^{m+n}$. The coring $A$ will be called \emph{strongly graded} if all of these maps are surjective, for all $m,\ n$. Correspondingly, the comultiplication $\Delta : C \to C \ot C$ of the $R$-coring $C$ has its components $\Delta_{m,n} : C_{m+n} \to C_m \ot C_n$ and $C$ is called \emph{strongly graded} if all of these maps are injective. Graded Sweedler's sigma notation for the image of an element $c \in C_{m+n}$ is:
\[ \Delta_{m,n}(c)=\sum c_{1,m} \ot c_{2,n}.\]

Also, the identity map of any set $M$ will be denoted by $\dr{Id}_M$.
	
\end{fact}

\begin{fact}[The Normalised (Co)Chain Complex.] \label{sec:bar}
For computing $T(A)$ and $E(C)$, respectively, we start by recalling the normalised (co)chain complexes that are required. We will be brief in doing this and we indicate \cite[\S 1.4]{ms} and \cite[\S 1.5 and \S 1.15]{jps} for further details.

To start off, $T_n(A) = \dr{Tor}_n^A(R,R)$ is the $n$-th homology group of the normalised bar complex $(\Omega_\bullet(A),d_\bullet)$, where $\Omega_n(A)=A_+^{(n)}$. The differentials $d_n : \Omega_n(A) \to \Omega_{n-1}(A)$ are defined as $d_1=0$ and for $n \geq 1$:
\[ d_n(a_1 \ot a_2 \ot \cdots \ot a_n)=\ds\sum_{i=1}^{n-1} (-1)^{i-1} a_1 \ot a_2 \ot \cdots \ot a_ia_{i+1} \ot \cdots \ot a_n.\]

It is well known the fact that the normalised bar complex is a DG coalgebra in the category of $R$-bimodules, thence $T(A)=\op_{n \geq 0} T_n(A)$ becomes canonically a connected, graded $R$-coring.

Dually, if $C$ is a graded connected $R$-coring, the corresponding normalised bar complex $(\Omega^\bullet(C),d^\bullet)$ has the terms $\Omega^n(C)=C_+^{(n)}$ and the differentials $d^n : \Omega^n(C) \to \Omega^{n+1}(C)$ defined as $d^0=0$ and for all $n \geq 1$:
\[ d^n = \ds\sum_{i=1}^n (-1)^{i-1} \dr{Id}_{C_+^{i-1}} \ot \Delta_+ \ot \dr{Id}_{C_+^{n-i}}.\]

Now, as in the dual case, $E^n(C)$ is the $n$-th cohomology group of this complex, computed as $\dr{Ext}^n_C(R,R)$ and $E(C) = \op_{n \geq 0} \ E^n(C)$ has a canonical structure of a connected, graded $R$-ring with respect to the concatenation of tensor monomials.
	
\end{fact}

\begin{fact}[The Shriek (Co)Ring.]
Using the notation and the notions from \cite[\S 1.2]{ms}, for and $R$-bimodule $M$ and a sub-bimodule $N \seq M \ot M$, define the $R$-ring $\langle M,N\rangle$ as being the quotient of the tensor algebra for the $R$-bimodule $M$ by the two-sided ideal generated by $N$, i.e. $T^a_R(M)/\langle N \rangle$. It is worth noting that $\langle N \rangle = \sum_{n \geq 0} \langle N \rangle ^n$, where this $n$-th degree component is obtained by:
\[ \langle N \rangle^n = \ds\sum_{i=1}^{n-1} M^{(i-1)} \ot N \ot M^{(n-i-1)}.\]	

Similarly, one could also define the graded $R$-coring $\{M,N\}$ by $\{M,N\}_0=R, \ \{M,N\}_1 = M$ and for all $n \geq 2$:
\[ \{M,N\}_n = \bigcap_{t=1}^{n-1} M^{(t-1)} \ot N \ot M^{(n-t-1)}.\]

We will be interested in a particular case of this construction. Namely, for a graded, connected $R$-ring $A$, consider the graded $R$-coring $\{A,\dr{Ker}\mu_{1,1}\}$, which will be denoted by $A^!$ and called \emph{the shriek coring of $A$}. Dually, for a graded connected $R$-coring $C$, consider $\langle C,\dr{Im}\Delta_{1,1}\rangle$, which we denote by $C^!$ and call \emph{the shriek ring of $C$}. For all positive integers $n$, the $n$-th degree components of these (co)rings will be denoted, for simplicity, by $A^!_n$ and $C^!_n$, respectively.

\end{fact}

\begin{fact}[The Koszul Complexes.] \label{sec:ksz-com}
As mentioned in the introduction, the Koszulity of an almost Koszul pair is conditioned by the exactness of any one of six Koszul complexes. We will present the construction of one of them, then explain how the others can be defined.

Let $(A,C)$ be an almost-Koszul pair. Define a complex of graded left $A$-modules $(\dr{K}^l_\bullet(A,C), \ d^l_\bullet)$ as follows. Let $\dr{K}_n^l(A,C) = A \ot C_n$, for all $n \geq 0$ and the maps $d_0^l(a \ot c) = \pi_A^0(a)c$, where $\pi_A^t : A \to A^t$ is the canonical projection on the homogeneous component of degree $t$. For $n \geq 1$, the differential is defined by the following equation:
\[ d_n^l(a \ot c) = \sum a \theta_{C,A}(c_{1,1}) \ot c_{2,n-1}.\]

Working over the opposite structure, i.e. considering the almost-Koszul pair $(A^{op},C^{op})$, one can construct a complex of graded right $A$-modules $(\dr{K}_\bullet^r(A,C), \ d_\bullet^r)$ whose $n$-th degree component is $\dr{K}_n^r(A,C) = C_n \ot A$. Also, these two complexes could be put together to form a complex of graded $A$-bimodules $(\dr{K}_\bullet(A,C),\ d_\bullet)$ for which $\dr{K}_\bullet(A,C) = A \ot C_n \ot A$ and $d_\bullet = d_n^l \ot \dr{Id}_A + (-1)^n \dr{Id}_A \ot d_n^r$.

By duality, one can construct three complexes of graded left (right, bi-) $C$-comodules and using \cite[Theorem 2.3]{jps}, infer that if any of the six complexes is exact, then all of them are so. If this is the case, the pair $(A,C)$ is called a \emph{Koszul pair}. 

Theorem 2.13 in \cite{jps}, Theorems 2.1 and 3.4 in \cite{ms} prove that this notion agrees with the ``classical'' (cf. \cite{bgs}) notion of Koszulity, in the sense that a Koszul pair consists of a Koszul $R$-ring $A$ and a Koszul $R$-coring $C$. Moreover, by \cite[Corollary 2.5]{jps}, the complexes provide appropriate resolutions for $A$ and $C$, respectively.
\end{fact}

\begin{fact}[Examples of (Almost) Koszul Pairs.] \label{sec:ex}
There are a few basic and important examples of (almost) Koszul pairs that are worth mentioning. These involve the shriek (co)rings that we introduced and also the $R$-coring $T(A)$ and the $R$-ring $E(C)$, for an (almost) Koszul pair $(A,C)$.

Thus, \cite[Proposition 1.18]{jps} proves that starting with any strongly graded, connected $R$-ring $A$, there is an almost Koszul pair $(A,T(A))$. As remarked in \S\ref{sec:bar}, $T(A)$ is a DG coring and the map $\theta_{T(A),A} : T_1(A) \to A^1$ is induced by the projection $A_+ \to A^1$ (note that $A$ being strongly graded implies that $T_1(A) = A_+/A_+^2 \simeq A^1$).

By duality, the same result shows that if $C$ is any strongly graded, connected $R$-coring, the pair $(E(C),C)$ is almost Koszul. The $R$-ring structure of $E(C)$ is presented in \S\ref{sec:bar} and the structural isomorphism $\theta_{C,E(C)} : C_1 \to E^1(C)$ acts as $\theta_{C,E(C)}(c) = c+ C_0 \in C/C_0 \simeq C_+$, since $E^1(C) = \dr{Ker}\Delta_+$.

The shriek structures also provide immediate examples of almost Koszul pairs. Using the fact that $A^!_1 = A^1$, take $\theta_{A^!,A} = \dr{Id}_{A^1}$ and note that $A^!_2 = \dr{Ker}\mu^{1,1}$, so the condition (\ref{ec:theta}) is trivially satisfied and the pair $(A,A^!)$ is almost Koszul.

Dually, the pair $(C^!,C)$ is almost Koszul as well with respect to the isomorphism $\theta_{C,C^!} = \dr{Id}_{C_1}$, which satisfies the equation (\ref{ec:theta}), as $C_2^! = (C_1 \ot C_1)/\dr{Im}\Delta_{1,1}$.

To end this section, we remark that Theorems 2.1 and 3.4 in \cite{ms} provide necessary and sufficient conditions for these pairs to be Koszul.
\end{fact}

\section{The $\dr{Ext}$ and Convolution Rings}

All of the $R$-rings and corings considered are left locally finite, so we can work with left graded duals without any further assumptions.

Thus, let $(A,C)$ be an almost Koszul pair, with $\theta := \theta_{C,A}$ being the structural isomorphism. Recall that $\dr{K}^l_\bullet(A,C) = A \ot C_\bullet$ is the Koszul complex (in the category of left $A$-modules) of the $R$-ring $A$ and $\beta^l_\bullet(A) = A \ot A_+^\bullet$ is the left normalised bar resolution for $A$. Define $\phi_\bullet : \dr{K}^l_\bullet(A,C) \to \beta_\bullet^l(A)$ as $\phi_{-1}=\dr{Id}_R$, the identity on $A \simeq A \ot C_0 = \dr{K}_0^l(A,C) = \beta_0^l(A)$ and for all $n \geq 1$ and $a \in A, \ c \in C_n$:
\[ \phi : \dr{K}_n^l(A,C) \to \beta_n^l(A), \qquad \phi_n(a \ot c) = \sum a \ot\theta(c_{1,1}) \ot \theta(c_{2,1}) \ot \dots \ot \theta(c_{n,1}).\]

With this definition, using \cite[Proposition 1.24]{jps} we conclude that $\phi_\bullet$ is a morphism of complexes, that lifts the identity of $R$.

Let $M$ be a left $R$-module. Recall that, by definition, we have $\Omega^n(A,M) = \dr{Hom}_R(A_+^n,M)$ and $\dr{K}_l^n(A,M) = \dr{Hom}_R(C_n,M)$, for any $n \geq 0$. Thus, by applying the contravariant functor $\dr{Hom}_A(-,M)$ to $\phi_\bullet$ and using the adjunction theorem, we get a morphism of complexes:

\[ \phi^\bullet : \Omega^\bullet(A,M) \to \dr{K}_l^\bullet(A,M),\]
which acts on $f : A^{(n)}_+ \to M$ and $c \in C_n$ as: 
\[\phi(f)(c) = f \big( \sum \theta(c_{1,1}) \ot \theta(c_{2,1}) \ot \dots \ot \theta(c_{n,1})\big).\]

%
%

The particular case that we are interested in is when $M=R$. Let us note that the differential $\partial^n$ of the first complex is null. Indeed, since the action of $A$ on $R$ is trivial, for $c \in C_{n+1}$ and $f \in \dr{Hom}_R(C_n,R)$, we have:
\[ \partial ^{n}(f)(c)=\sum\limits \theta
(c_{1,1})f(c_{2,n}).\]
In particular, taking the cohomology it follows that $\phi^\bullet$ induces a map from $\lgrr C = \op_{n\geq 0} \ \dr{Hom}_R(C_n,R)$ to $\dr{Ext}^\bullet_A(R,R) = \kal{E}(A)$.  

Now we can prove the following result.

\begin{proposition} \label{te:yon}
Let $(A,C)$ be an almost Koszul pair. There exists a canonical morphism of graded $R$-rings from $\kal{E}(A) = \dr{Ext}^\bullet_A(R,R)$ with the cup product to the left graded dual of $C$ with the graded convolution product.
\end{proposition}

In order to show this, we first proceed with a lemma.

\begin{lemma} \label{lema-dgcor}
\begin{enumerate}[(a)]
\item Let $C$ be a differential graded (DG) coring. Then $\lgrr C$ is a DG ring.
\item If $f : C \to D$ is a morphism of DG corings, then its transpose ${^{\ast}} f : \lgrr D \to \lgrr C$ is a morphism of DG rings.
\end{enumerate}

In other words, $\lgrr (-)$ is a functor from the category of DG corings to the category of DG rings.
\end{lemma}

\begin{proof}
(a) We know by \cite[\S 4.2]{ms} that the left graded dual of an $R$-ring is an $R^{op}$-coring. Thus, we are left with proving that the multiplication of $\lgrr C$ is compatible with the differential. Hence, if $d$ denotes the differential of $C$, then ${^\ast}d$, its transpose, is the differential of $\lgrr C$. We know that $d$ is a coderivation, hence $\Delta d(c) =\sum d (c_{1,n})\ot c_{2,m} +(-1)^n\sum c_{1,n} \ot d(c_{2,m})$, where $\Delta$ is the comultiplication of $C$ and $c\in C_{n+m+1}$. We have to show that:
\[ (^\ast d)(\alpha \ast \beta) = ({^\ast} d\alpha) \ast \beta + (-1)^n \alpha \ast ({^\ast}d\beta),\]
for all $\alpha \in \dr{Hom}_R(C_n,R), \ \beta \in \dr{Hom}_R(C_m,R)$. Furthermore, we recall the convolution product of $\lgrr C$ that is defined by:
\[ (\alpha \ast \beta)(c) = \sum \alpha(c_{1,n}  \beta(c_{2,m})),\] 
for any $\alpha,\beta$ as above and $c \in C_{n+m}$.

Indeed, for $c\in C_{n+m+1}$, using the notations above, we compute the following:
\[ (({^\ast}d \alpha) \ast \beta)(c) = \sum {^\ast}d\alpha(c_{1,n+1}\beta(c_{2,m})) = \sum\alpha(d(c_{1,n+1})\beta(c_{2,m})).\]
The other summand is:
\[ (\alpha \ast ({^\ast}d\beta))(c) = \sum \alpha(c_{1,n}({^\ast}d\beta)(c_{2,m+1}))) = \sum \alpha (c_{1,n} \beta(d(c_{2,m+1}))).\]
Finally,
\[({^\ast}d(\alpha \ast \beta))(c) = (\alpha \ast \beta)(dc) = \sum \alpha(d(c_{1,n+1})\beta(c_{2,m})) + (-1)^n \sum \alpha (c_{1,n} \beta(d(c_{2,m+1}))).\]
The first equality follows by the definition of ${^\ast}d$ and the second, by the coderivation property of $d$ and the definition of the convolution product. 

(b) The second part of the lemma is easily proved using the definitions. We already know that the transpose of a graded coring map is a graded ring map, so we have to show that it is also compatible with the differentials. Let $d_C$ and $d_D$ be the corresponding differentials of the DG corings $C$ and $D$, respectively. Then $f$ being a morphism of DG corings implies that $f\circ d_C = d_D \circ f$. Using only the definitions of the transposed maps we obtain that ${^\ast f} \circ {^\ast d_D} = {^\ast}d_C \circ {^\ast}f$, as needed.
\end{proof}

Now we can prove the Proposition \ref{te:yon}. 

\begin{proof}
Recall that there is a morphism of complexes:
\[ \phi_\bullet : \dr{K}_\bullet^l (A,C) \to \beta_\bullet^l(A,C).\] 

By deleting the component of degree -1 and applying the functor $(-) \ot_A R$, this induces a morphism (which we will still denote by $\phi_\bullet$) of DG corings:
\[ \phi_\bullet : \dr{K}^l_\bullet(A,R) \to \Omega_\bullet(A).\]

But as remarked previously, the first DG coring is $C_\bullet$ with the trivial differential. Then, by the first part of Lemma \ref{lema-dgcor}, it follows $\lgrr C$ is a DG ring. The second part of the Lemma \ref{lema-dgcor} and the identification $\lgrr \Omega_\bullet(A) = \Omega^\bullet(A,R)$ complete the proof by passing to cohomology.
\end{proof}

%
\begin{theorem} \label{cor:yon}
If the pair $(A,C)$ is Koszul, then $\phi^\bullet$ is invertible and it induces an isomorphism
\[ \kal{E}(A) \simeq \lgr T(A).\]
\end{theorem}

\begin{proof}
Koszulity of the pair $(A,C)$ ensures that the map $\phi_\bullet$ is a quasi-isomorphism, as it is obtained from a map between the Koszul and bar resolutions, which lifts the identity of $R$. Hence the morphism $\phi^\bullet = \dr{H}_\bullet(\phi_\bullet,R)$ is bijective.
Using \cite[Theorem 2.9]{jps}, when $(A,C)$ is a Koszul pair, $C \simeq T(A)$ and $A \simeq E(C)$. Since both of the structures are left locally finite, taking left graded duals preserves the isomorphisms. Thus, we obtain the required result. \end{proof}

\begin{corollary} \label{cor:eea}
The isomorphism in the theorem can be further expanded to $\kal{E}(A) \simeq E(\lgrr A)$.
\end{corollary}

\begin{proof}
If $(A,C)$ is a Koszul pair, using \cite[Corollary 4.7]{ms}, we know that $\lgr T(A) \simeq E(\lgrr A)$. Using this isomorphism in the theorem, by transitivity, we obtain $\kal{E}(A) \simeq E(\lgrr A)$, as needed.
\end{proof}

We can formulate yet another result that characterises Koszulity in terms of some properties of the $\dr{Ext}$ ring. In fact, in a more general framework, the following result holds true:

\begin{lemma}
Let $R$ be a semisimple ring and $V, \ W$ be $R$-bimodules, which are finitely generated as left $R$-modules. A morphism of $R$-bimodules $f : V \to W$ is injective if and only if its left dual ${^\ast}\! f = \dr{Hom}_R(f,R)$ is surjective.
\end{lemma}

\begin{proof}
Let $X$ be the kernel of $f$, so the sequence $0 \to X \to  V \xrightarrow{f} W$ is exact. Since $R$ is a semisimple ring, it follows that $R$ is an injective $R$-module. Hence the functor $\dr{Hom}_R(-,R)$ is contravariant exact and the sequence $^\ast W \xrightarrow{^\ast \! f} {}^\ast V \to {}^\ast\! X \to 0$ is also exact. It follows that $^\ast\! X$ is the cokernel of $f$.

Now we can make a remark: $X = 0$ if and only if $^\ast\! X=0$. Indeed, if $X=0$, the result holds trivially. Conversely, assume that $^\ast\! X=0$. We will prove that if $X\neq 0$, then we can construct a nonzero left $R$-linear map from $X$ to $R$, thus obtaining a contradiction.

Hence, suppose that there exists $x \neq 0$ an element in $X$. There is an injective mapping $Rx \to X$ and we will use the fact that $R$ is an injective $R$-module to extend this to a map $X \to R$. Let $\varphi$ be the map $R \to Rx$ that sends $1$ to $x$. Then $\dr{Ker}\varphi$ is a left $R$-submodule of $R$ and $R/\dr{Ker}\varphi \simeq Rx$. 

Note that if $\dr{Ker}\varphi = R$, then $\varphi = 0$, but this contradicts the isomorphism of $R / \dr{Ker}\varphi$ with $Rx$, as $x$ was assumed to be nonzero. 

For $\dr{Ker}\varphi \neq R$, since $R$ is semisimple, $\dr{Ker}\varphi$ is a direct summand of $R$, hence there exists a left $R$-submodule $J$ such that $R \simeq \dr{Ker}\varphi \op J$. Then we have the following:
\[ Rx \simeq R/\dr{Ker}\varphi \simeq J \hookrightarrow R,\]
so we can define a map $\psi : Rx \to R$ as the composition of the last isomorphism and the inclusion.

Thus, using the fact that $R$ is an injective left $R$-module, we can complete the diagram:
\[
\xymatrixcolsep{3pc}\xymatrix{
0 \ar[r] & Rx \ar[d]_-\psi \ar[r] & X \ar@{.>}[dl]^-g \\
 & R & }
\]
The nonzero map $g : X \to R$ provides a contradiction and completes the proof.
\end{proof}
We will now make use of this result in the context of the left graded dual of an $R$-coring $A$. We refer the reader to \cite[\S 4.1-4.2]{ms} for the construction and the basic properties of such structures. For any $n, \ m \in \mathbb{N}$, consider the diagram:
\[ \xymatrixcolsep{3pc}\xymatrix{
{^\ast\!} A_{n+m} \ar[r]^-{{}^\ast\mu_{n,m}} \ar[dr]_-{\Delta_{n,m}} & {^\ast\!} (A_n \ot A_m) \ar[d] \\
& {^\ast\!} A_n \ot {^\ast\!} A_m & }\]

The vertical arrow is an isomorphism and ${}^\ast\mu_{n,m}$ is the transpose of the multiplication of $A$. Also, the map $\Delta_{n,m}$ is the comultiplication of the $R^{op}$-coring $\lgrr A$. Note also that if $A^p$ is a finitely generated $R$-bimodule, for all $p \in \bb{N}$, then $A^n \ot A^m$ is also finitely generated, for all positive integers $n,\ m.$

Using the lemma above, we can prove immediately the following:

\begin{proposition}
The $R$-ring $A$ is strongly graded if and only if the $R^{op}$-coring $\lgrr A$ is strongly graded.
\end{proposition}
\begin{proof}
Looking at the diagram above, it suffices to note that $A$ is strongly graded if and only if $\mu^{n,m}$ is surjective, for all $n, \ m \in \bb{N}$. This holds if and only if ${}^\ast\mu_{n,m}$ is injective and $\Delta_{n,m}$, which is its composition with an isomorphism, is also injective. This, in turn, according to the definition, makes the $R^{op}$-coring $\lgrr A$ strongly graded.
\end{proof}

In the hypothesis of $A$ being a Koszul ring, we can provide another isomorphism that connects the $\dr{Ext}$ and the shriek structures.

\begin{theorem} \label{te:shriek}
Let $A$ be a connected, left locally finite $R$-ring. If $A$ is Koszul, then $\kal{E}(A)$ is a Koszul $R$-ring and there exists an isomorphism $\kal{E}(A) \simeq (\lgrr A)^!$.
\end{theorem}
\begin{proof}
Using \cite[Theorem 2.1]{ms}, $A$ is a Koszul $R$-ring if and only if $T(A)$ is a Koszul $R$-coring. Further, cf. \cite[Corollary 4.5]{ms}, this is equivalent to $\lgr T(A)$ being a Koszul $R$-ring. Now the isomorphism $\kal{E}(A) \simeq \lgr T(A)$ from Theorem \ref{cor:yon} makes $\kal{E}(A)$ a Koszul $R$-ring. Moreover, since $\lgrr A$ is Koszul, using \cite[Theorem 3.4]{ms}, we know that $E(\lgrr A) \simeq (\lgrr A)^!$. Using Corollary \ref{cor:eea}, the required isomorphism follows.
\end{proof}

Note that if $\kal{E}(A)$ is a Koszul $R$-ring, then, in particular, it is strongly graded. This means that it is generated by its homogeneous component of degree 1 or, equivalently, that the components of the cup product $\cup^{p,q} : \kal{E}^p(A) \ot \kal{E}^q(A) \to \kal{E}^{p+q}(A)$ are surjective, for all $p, \ q \in \bb{N}$.

The interplay between Koszulity and the $\dr{Ext}$ ring can be further put together in a simple, yet comprehensive corollary:

\begin{corollary}
If $A$ is a Koszul $R$-ring then $T(A)$ is a strongly graded $R$-coring and $\kal{E}(A)$ is a strongly graded $R$-ring.
\end{corollary}

\section{Applications}

\subsection{Koszul Posets}
Let $\kal{P}$ be a finite partially ordered set (poset). We call $\kal{P}$ \emph{graded} if all of the maximal chains in $\kal{P}$ have the same length. It is a basic fact that $\kal{P}$ is graded if and only if the incidence algebra of $\kal{P}$ is graded (hence the name).

Let $\Bbbk$ be a field. We are interested in applying the results in the previous section for $A=\Bbbk^a[\kal{P}]$, the incidence algebra of $\kal{P}$ which is an $R$-ring with respect to the semisimple ring $R=\Bbbk^{\#\kal{P}}$. We will also consider $\Bbbk^c[\kal{P}]$, which is the incidence $R$-coring of the poset $\kal{P}$. Call the poset $\kal{P}$ \emph{Koszul} when $\Bbbk^a[\kal{P}]$ is a Koszul $R$-ring (or, equivalently, when $\Bbbk^c[\kal{P}]$ is a Koszul $R$-coring). For more details on these structures, we indicate \cite[\S4.8]{ms}.

Now we can apply Theorem \ref{te:shriek} in this context.

\begin{proposition}
Let $(\kal{P},\leq)$ be a Koszul poset. Let $A=\Bbbk^a[\kal{P}]$ be its incidence $R$-ring and $C=\Bbbk^c[\kal{P}]$ be its incidence $R$-coring. Then there is an isomorphism of graded $R$-rings:
\[ \kal{E}(\Bbbk^a[\kal{P}]) \simeq \Bbbk^c[\kal{P}]^!,\]
where the former is the shriek ring associated to the incidence $R$-coring of $\kal{P}$.
\end{proposition}
\begin{proof}
Using Theorem \ref{te:shriek} above, note that, by \cite[Theorem 4.9]{ms}, we have the isomorphism $\Bbbk^a[\kal{P}] \simeq \lgr \Bbbk^c[\kal{P}]$. Since $\kal{P}$ is Koszul, it follows that $(A,C)$ is a Koszul pair. Using \cite[Theorem 2.1]{ms}, we know that the pair $(A,A^!)$ is Koszul as well and that there exists a canonical isomorphism $A^! \to T(A)$. Moreover, $T(A) \simeq C$ and the pair $(C^!,C)$ is also Koszul (cf. \cite[Theorem 3.4]{ms}). Putting these together in Theorem \ref{cor:yon}, we obtain the desired result.
\end{proof}

This result brings a simplification for computing the $\dr{Ext}$ structure of an incidence ring. It will provide a method of presenting the structure in a generators and relations form. This can be done using an isomorphism with a shriek ring, as follows. Let $\Gamma$ be the finite quiver associated to the poset $\kal{P}$ that is its Hasse diagram. Then, via the isomorphism in the proposition above, we have the following explicit description:
\begin{equation} \label{ec:ka}
 \kal{E}(\Bbbk^a[\kal{P}]) \simeq \Bbbk^c[\kal{P}]^! = \ds\frac{\Bbbk^a[\Gamma]}{\Big\langle \ds\sum_{x < z < y} e_{x,z} \ot e_{z,y} \ \Big| \ l([x,y])=2 \Big\rangle}.	
\end{equation}

\begin{figure}
\centering
\begin{subfigure}[b]{0.3\textwidth}
\includegraphics[width=\textwidth]{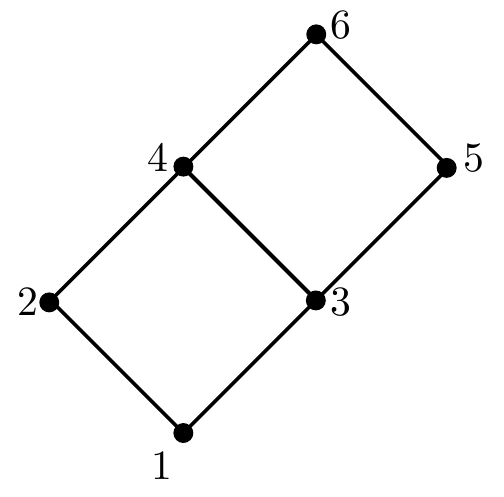}
\caption{Planar Koszul Tiling}	
\label{fig1}
\end{subfigure}
\qquad
\begin{subfigure}[b]{0.3\textwidth}
\includegraphics[width=\textwidth]{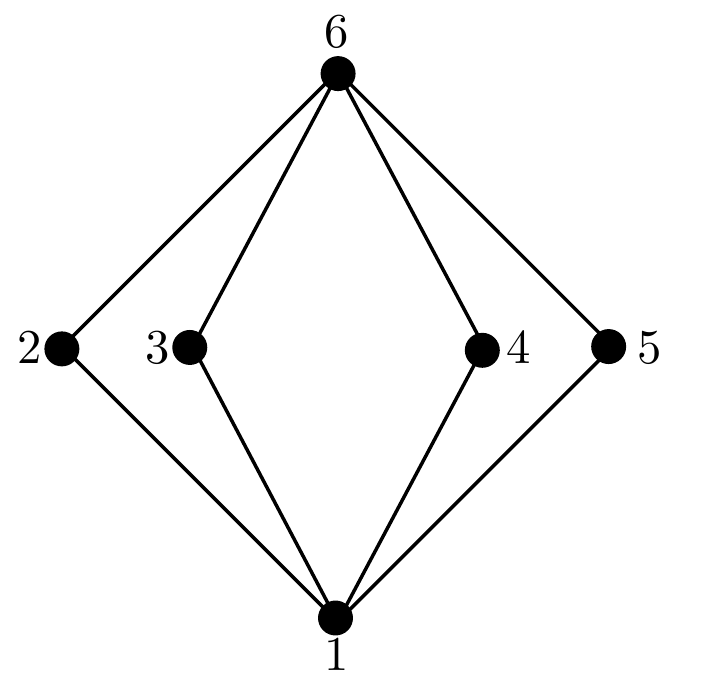}
\caption{Nested Koszul Diamond}	
\label{fig2}
\end{subfigure}
\caption{Concrete Examples of Koszul Posets}
\end{figure}

For example, the poset in Figure {\color{red} 1(A)} is Koszul, by \cite[\S 5]{ms}. 
Thus, the incidence algebra of the corresponding quiver, $A = \Bbbk^a[\Gamma]$ is generated by the set $\kal{B}=\{e_{1,2}, \ e_{1,3}, \ e_{2,4}, \ e_{3,4}, \ e_{3,5}, \ e_{4,6}, \ e_{5,6}\}$.
The relations are given by the $\Bbbk$-vector space $\kal{R}$ spanned by parallel paths (i.e. sharing the source and the target) of length two. Namely: $ \kal{R} = \langle e_{1,2} \ot e_{2,4} + e_{1,3} \ot e_{3,4}, \ e_{3,4} \ot e_{4,6} + e_{3,5} \ot e_{5,6} \rangle$.

Therefore, the identification in equation (\ref{ec:ka}) gives the isomorphism: $\kal{E}(\Bbbk^a[\kal{P}]) \simeq \Bbbk^c[\kal{P}]^! = \Bbbk^a[\Gamma] / \kal{R}$.

Another example that we consider is that of a nested vertical diamond. \cite[\S 5.14]{ms} proves that these posets are also Koszul. As such, consider the poset in Figure {\color{red} 1(B)}.
A similar reasoning as above provides the generator set $\kal{B} = \{ e_{1,2}, \ e_{1,3}, \ e_{1,4}, \ e_{1,5}, \ e_{2,6}, \ e_{3,6}, \ e_{4,6}, \ e_{5,6}\}$. In this case, the ideal of relations is generated by one element:
\[\kal{R} = \Big\langle \sum_{i=2}^5 e_{1,i} \ot e_{i,6} \Big\rangle.\]

\subsection{Monoid Rings}
We can provide another type of examples for Koszul rings, taken from the class of monoid rings and using the special case of submonoids of $\bb{Z}^n$, the $n$-fold product of the integer set. We will partly follow the approach from \cite[\S 1.2]{str}. The basic setup is as follows. Consider $(M,+)$ an associative and cancellative monoid. By definition, this means that if $a + b = a + d$ for some arbitrary elements $a,b,d \in M$, then $a$ can be cancelled and $b=d$. Analogously, if $b + a = d + a$, then $b=d$. Let $\Bbbk$ be a field. The \emph{monoid ring} $\Bbbk M$ associated to $M$ has a $\Bbbk$-vector space structure with a basis given by the elements $\{\xi_m\}_{m \in M}$ indexed on $M$ and multiplication defined as $\xi_m \cdot \xi_n = \xi_{m+n}$, extended $\Bbbk$-bilinearly.

To study Koszulity, consider $A=\Bbbk M$, which could be decomposed as $A = \Bbbk\xi_0 \op A_+$, where $A_+ = \Bbbk \xi_m, \ m \neq 0$. This way $A$ becomes an augmented $\Bbbk$-algebra. 

Let $A^1$ be the $\Bbbk$-span of all the indecomposable elements of $M$. That is, those $m \in M$ which cannot be written as $m = m_1 + m_2$, with $m_1, \ m_2 \in M$, both nonzero. It is known that $A$ is generated as an $\Bbbk$-algebra by $A^1$ and it becomes $\bb{N}$-graded if and only if there exists a functional on $\bb{R}^n$ which takes the value 1 on all indecomposables $m \in A^1$. The homogeneous component of degree $n$ is generated by elements of the form $\xi_m$, where $m$ can be written as a sum of exactly $n$ indecomposable elements of the monoid. In this case, we write $deg(m)=n$.

We can describe the $\dr{Ext}$ algebra of a monoid ring $A=\Bbbk M$ as follows. Assume that $A$ has a finite number of indecomposable elements. Thus $A^n=\langle \xi_m \mid \ deg(m)=n \rangle$, hence $A$ is finitely generated. Now, the graded linear dual $A^{\ast\text{-gr}}$ has a dual basis over the field $\Bbbk$ in the usual sense. That is, $A_n^\ast=\langle \xi^\ast_m \mid \ deg(m) = n \rangle$. The comultiplication $\Delta : A^{\ast\text{-gr}}\to A^{\ast\text{-gr}} \ot A^{\ast\text{-gr}}$ acts as:
\[
\Delta(\xi^\ast_m) = \sum_{m} \xi^\ast_{m'} \ot \xi^\ast_{m''}.
\]

In the sum above, we have used a Sweedler-type notation, in the sense that the indices $m'$ and $m''$ are arbitrary and sum up to $m$.

Then, using the isomorphism in Theorem \ref{cor:yon}, we can provide a presentation of $\kal{E}(A)$ with generators and relations. Namely, $\kal{E}(A)$ is the quotient of the tensor algebra of ${}^\ast\! A^1$ modulo the ideal generated by the image of $\Delta_{1,1} : {}^\ast\! A^2 \to {}^\ast\! A^1 \ot {}^\ast\! A^1$. Note that for $deg(m)=2$, $\Delta_{1,1}(\xi^\ast_m)=\sum \xi^\ast_{m'} \ot \xi^\ast_{m''}$, where $m'$ and $m''$ are indecomposable and $m' + m''=m$. Thus, we obtain:
\begin{equation} \label{ec:shriek}
\kal{E}(A) = \ds\frac{\Bbbk[\xi^\ast_m \mid deg(m)=1]}{\Big\langle \Delta_{1,1}(\xi^\ast_m) \mid \deg(m)=2 \Big\rangle}.
\end{equation}

As mentioned, Koszulity of monoid rings (and semigroup rings, in general) is studied in relation to posets in \cite{str}. We will restrict here to studying a concrete example with a slightly different approach.

Let $M$ be the submonoid of $\bb{Z}^2$ generated by the elements $m_1=(2,0), \ m_2=(0,2), \ m_3=(1,1)$. These elements are indecomposable in $M$, so if $A=\Bbbk M$, then $A^0 = \Bbbk \xi_0 \simeq \Bbbk$ and $A^1$ is generated as a $\Bbbk$-vector space by the set $\{m_1, \ m_2, \ m_3\}$. It is clear that for this particular case, $A$ is an $\bb{N}$-graded $\Bbbk$-algebra. 

Moreover, we can provide a presentation of $A$ with generators and relations. Note that $A$ is generated by $A^1$ and has as unique relation $\xi_{m_1} \cdot \xi_{m_2} = \xi_{m_3}^2$. Thus, we can identify $A$ with a quotient of a polynomial algebra, namely $A \simeq \Bbbk[X,Y,Z]/(X^2 - YZ)$. The fact that this is a Koszul $R$-ring can be seen using \cite[Corollary 6.3]{pp}.

Now we can compute explicitly the $\dr{Ext}$ algebra of $A$. Using the presentation in equation (\ref{ec:shriek}), make a shorthand notation $\xi^\ast_{m_i} = x_i$ and we obtain readily:
\[ \kal{E}(A) = \ds\frac{\Bbbk \langle x_1,x_2,x_3  \rangle}{\langle x_1^2,x_2^2,x_3^2 + x_1x_2 + x_2x_1, x_1x_3 + x_3 x_1 \rangle}.\]

\section*{Acknowledgements}
The author was financially supported by the Ministry of Education - OIPOSDRU, project POSDRU/159/1.5/S/137750. 

Gratitude is also due to Drago\c{s} \c{S}tefan and Marian Aprodu for many useful discussions on the subject.

\end{document}